\documentclass[draft]{amsart}
\numberwithin{equation}{section}
\numberwithin{figure}{section}
\usepackage{graphicx}
\usepackage{amssymb}
\let\cal\mathcal

\def\Rscr{{\cal R}}
\def\Sscr{{\cal S}}

\let\blb\mathbb

\def \ZZ{{\blb Z}}

\def \RR{{\blb R}}

\def\Hom{\operatorname {Hom}}

\def\r{\rightarrow}

\newtheorem{lemma}{Lemma}[section]
\newtheorem{proposition}[lemma]{Proposition}
\newtheorem{theorem}[lemma]{Theorem}
\newtheorem{corollary}[lemma]{Corollary}

\newtheorem*{remark*}{Remark}

\theoremstyle{definition}

\newtheorem{example}[lemma]{Example}
\newtheorem{definition}[lemma]{Definition}

{

}

\theoremstyle{remark}

\newtheorem{remark}[lemma]{Remark}

\newdimen\uboxsep \uboxsep=1ex
\def\uboxn#1{\vtop to 0pt{\hrule height 0pt depth 0pt\vskip\uboxsep
\hbox to 0pt{\hss #1\hss}\vss}}

\def\uboxs#1{\vbox to 0pt{\vss\hbox to 0pt{\hss #1\hss}
\vskip\uboxsep\hrule height 0pt depth 0pt}}

\input{conjecture_pictures_tikz}

\title{Subgroups of braid groups generated by Birman-Ko-Lee generators}
\author{Anya Nordskova}
\author{Michel Van den Bergh}
\thanks{The second author is a senior researcher at the Research
  Foundation Flanders (FWO). This project has received funding from the European Research Council (ERC) under the European Union's Horizon 2020 research and innovation programme (grant agreement No 885203). Part of this paper was conceived while the second author was in residence at the Simons Laufer Mathematical Sciences Institute (formerly MSRI) in Berkeley, California, during the Spring 2024 semester.}
\address[Anya Nordskova and Michel Van den Bergh]{Vakgroep Wiskunde, Universiteit Hasselt, Universitaire Campus \\
  B-3590 Diepenbeek}
\email{anya.nordskova@uhasselt.be}
\email{michel.vandenbergh@uhasselt.be}
\address[Michel Van den Bergh]{Vakgroep Wiskunde en Data Science, Vrije Universiteit Brussel, Pleinlaan 2, 1050 Brussel} 
\email{michel.van.den.bergh@vub.be}
\begin{document}
\begin{abstract}
We define a Young subgroup of the braid group as a subgroup generated by an arbitrary subset of the Birman-Ko-Lee generators.
We give an intrinsic description of such subgroups which yields, in particular,
an easy criterion to decide membership. We also give an algorithm to write an element of a Young subgroup
as a product of the generators. Our methods are based on analyzing the Hurwitz action on tuples over
free groups via a diagrammatic approach.
\end{abstract}
\maketitle
\section{Introduction}
\subsection{Statement of results}
Let $B_n$ be the $n$-strand braid group with the usual generators
$\sigma_1,\ldots,\allowbreak \sigma_{n-1}$. It is standard that any
subgroup of $B_n$ generated by a subset of these generators is again
a product of braid groups.  
More precisely such a subgroup is of the form $B_Q:=\prod_{i} B_{Q_i}$ where $Q=(Q_i)_i$
is a partition of $\{1,\ldots,n\}$ consisting of
intervals and $B_{Q_i}$ stands for the braid group whose strands are labeled by the elements of $Q_i$.

\medskip

Under the canonical morphism $B_n\r S_n$ the generator $\sigma_i$ maps to the transposition $(i,i+1)$. In the celebrated paper \cite{BKL} the authors define certain
``dual'' generators $(a_{ij})_{i<j}$ for $B_n$ which map to the transposition $(i,j)$. Formally $a_{ij}$ 
is defined by
\[
a_{ij}=\sigma_i\cdots\sigma_{j-2}\sigma_{j-1}\sigma^{-1}_{j-2}\cdots\sigma_i^{-1}.
\]
To clarify this unenlightening definition it is best to identify $B_n$
with the mapping class group of $(\RR^2,\{p_1,\ldots,p_n\})$. If we place $p_1,\ldots,p_n$
on the horizontal axis then $\sigma_k\in B_n$ may be identified with
the half twist (see \S\ref{sec:halftwist}) associated to the interval
$\Sigma_k=[p_k,p_{k+1}]$  (we orient $\RR^2$ clockwise).  With this convention $a_{ij}\in B_n$ is
equal to the half twist associated to the lower {arc} $A_{ij}$
connecting $p_i$ to $p_j$ as in Figure \ref{fig:Aij}.
\begin{figure}[ht]
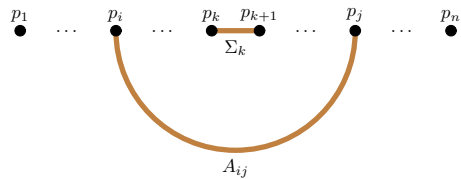

\disk
\caption{The Birman-Ko-Lee generators $a_{ij}$}
\label{fig:Aij}
\end{figure}

Since $\sigma_i=a_{i,i+1}$ it makes sense to generalize the setting considered in the first paragraph
by investigating subgroups of $B_n$ generated by a subset of the generators $(a_{ij})_{ij}$. 

From the mapping class group interpretation, or directly from the relations among the $a_{ij}$ (see e.g.\ \cite[Lemma VIII.1.3]{MR2463428}) one deduces that any
element of $\{a_{ij},a_{jk},a_{ik}\}$ can be expressed in terms of the two others.
If follows that we only need to consider sets of generators associated to a partition of $\{1,\ldots,n\}$.
So if $Q=\{Q_1,\ldots,Q_m\}$ is such a partition then we define the \emph{Young subgroup} $B_Q$ of~$B_n$ as the subgroup generated by those
$a_{ij}$ for which $i,j$ are in the same component of~$Q$. This extends our earlier use of the notation $B_Q$ when $Q$ consisted of a collection of intervals.

We say that a partition $Q$ is \emph{non-crossing} if the {arc}s in Figure \ref{fig:Aij} which connect points in the same partition do not intersect for different partitions
or, in other words, if we do not have 
$
i<j<k<l$ with $i,k\in Q_s$, $j,l\in Q_t$ with $s\neq t$.
If~$Q$ is non-crossing then   $a_{ij}a_{kl}=a_{kl}a_{ij}$ for all pairs $\{i,j\}$ and $\{k,l\}$ contained in different components of $Q$ and thus it follows that, as in the first paragraph, we have
\[
B_Q=\prod_{i=1}^m B_{Q_i}.
\]
\begin{remark}
In this case $B_Q$ is a so-called {\em standard parabolic subgroup of $B_n$} in the sense of \cite[Definition 2.2]{Godelle}.
\end{remark}
However if $Q$ is not non-crossing
then $B_Q$ is an altogether different beast as can be inferred from Lemma \ref{lem:norel} which states that if $i<j<k<l$
or equivalently if $A_{ik}$ and $A_{jl}$ intersect in Figure \ref{fig:Aij} then $a_{ik}$ and $a_{jl}$ generate a free subgroup of $B_n$.

\medskip

In this paper we study the groups $B_Q$ for arbitrary partitions $Q$. In particular we solve the following two problems:
\begin{enumerate}
\item[(P1)] We give a more intrinsic description of $B_Q$ which allows one to decide if a given element of $B_n$ is in $B_Q$.
\item[(P2)] If $b\in B_Q$ then we give an algorithm to write $b$ in terms of the generators of $B_Q$.
\end{enumerate}
We now explain this. Recall that if $G$ is a group then $B_n$ acts on $G^n$ via the ``Hurwitz action'':
\[
\sigma_i(g_1,\ldots,g_i,g_{i+1},\ldots,g_n)=(g_1,\ldots,g_{i+1}, g_{i+1}^{-1}g_i g_{i+1},\ldots,g_n)\,.
\]
One checks that if $g_i=g_j$ then $a_{ij}$ leaves $(g_i)_i$ invariant.
\medskip

For a partition $Q$ with $m$ parts put $F_m=\langle x_1,\ldots,x_m\rangle$ and define $t\in F_m^n$ in such
a way that $t_i=x_s$ iff $i\in Q_s$. We prove:
\begin{theorem}[Theorem \ref{th:mainth}\eqref{it:mainth2}] \label{eq:stab}
$B_Q$ is equal to the stabilizer of $t$ in $F_m^n$.
\end{theorem}
\begin{example} \label{ex:usual}
Assume $m=2$ and $F_2=\langle x,y\rangle$. Put $t=(x,y,x,y)$. Then the stabilizer of $t$ is generated by $a_{13} = \sigma_1\sigma_2\sigma_1^{-1}$ and $a_{24} = \sigma_2\sigma_3\sigma_2^{-1}$. These elements generate a free subgroup of $B_4$ by Lemma \ref{lem:norel}.
\end{example}
Theorem \ref{eq:stab} immediately solves (P1): to check that $b\in B_n$ is in $B_Q$  we only need to verify that $bt=t$.
For the algorithm alluded to in (P2) we refer to \S\ref{sec:algo}. 

\medskip

Besides knowing the stabilizer of $t$ it is of course also interesting to know its orbit. This orbit is described in the following theorem.
\begin{theorem}[Theorem \ref{th:mainth}\eqref{it:mainth1}] The orbit of $t = (t_1, \dots,t_n)$ is given by the set of elements in $\bigcup_{\tau \in S_n} C(t_{\tau(1)})\times\cdots\times C(t_{\tau(n)})$ whose product is $t_1\cdots t_n$
(where $C(x)$ denotes the conjugacy class of $x$).
\end{theorem}
Although it is not used in the statements of our results, our proof makes essential use of a procedure to associate certain arc diagrams to elements of 
$C(t_{\tau(1)})\times\cdots\times C(t_{\tau(n)})$. This is explained in \S\ref{sec:arcdiagramfree}.
\subsection{Motivation}
The results in this paper will be used elsewhere to describe the orbit and the stabilizer of certain exceptional collections  under the mutation action~\cite{MR992977} by the braid group.

\section{Preliminaries and statement of the main results}
\subsection{The Hurwitz action}
Below $B_n$ stands for the braid group with standard generators  $\sigma_1,\ldots,\sigma_{n-1}$
and $F_m$ stands for the free group
generated by $x_1,\ldots,x_m$. The group $B_n$ acts on $F_n$ via
\[
\sigma_i(x_j)=
\begin{cases}
x_{i}x_{i+1}x_i^{-1}&\text{if $j=i$},\\
x_{i}&\text{$j=i+1$},\\
x_j&\text{$j\not\in\{i,i+1\}$}.
\end{cases}
\]
For convenience we give the formula for $\sigma_i^{-1}$.
\[
\sigma_i^{-1}(x_j)=
\begin{cases}
x_{i+1}&\text{if $j=i$},\\
x_{i+1}^{-1}x_{i}x_{i+1}&\text{$j=i+1$},\\
x_j&\text{$j\not\in\{i,i+1\}$}.
\end{cases}
\]
If $G$ is a group then a \emph{Hurwitz system} of length $n$ is a group
homomorphism $\phi:F_n\r G$.  We identify such a Hurwitz system 
with the element
of $G^n$ given by
$(\phi(x_1),\ldots,\phi(x_n))\in G^n$.

Since $B_n$ acts on $F_n$, it also acts on $\Hom_{\operatorname{Groups}}(F_n,G)$ by the usual formula
\[
(b\cdot \phi)(w)=\phi(b^{-1}w).
\]
for $b\in B_n$, $w\in F_n$, $\phi\in \Hom_{\operatorname{Groups}}(F_n,G)$.
  The corresponding action on $G^n$ is referred to as the \emph{Hurwitz action} and is given by
\begin{align*}
\sigma_i(g_1,\ldots,g_i,g_{i+1},\ldots, g_n)&=(g_1,\ldots,g_{i+1}, g^{-1}_{i+1} g_i g_{i+1},\ldots, g_n),\\
\sigma_i^{-1}(g_1,\ldots,g_i,g_{i+1},\ldots, g_n)&=(g_1,\ldots,g_i g_{i+1}g_i^{-1},g_i,\ldots, g_n).
\end{align*}
If $b\in B_n$ then we denote the corresponding element of $S_n$ by $\bar{b}$. $S_n$ acts on $F_n$ by permutation of the
variables, i.e. $\sigma(x_i)=x_{\sigma(i)}$. The corresponding action on  $G^n\cong \Hom(F_n,G)$ is given by $\sigma\cdot (g_1,\ldots, g_n)=(g_{\sigma^{-1}(1)},\ldots, g_{\sigma^{-1}(n)})$.

The following lemma recalls some trivial properties of the Hurwitz action.
\begin{lemma}
\label{lem:trivial_hurwitz}
\begin{enumerate}
\item If $\mu:G^n\r G$ denotes the product map then for $b\in B_n$, and $g\in G^n$ we have $\mu(b g)=\mu (g)$.
\item If $b\in B_n$ then the action of $b$ on $G^n$ sends $C(g)$ to $C(\bar{b}(g))$  where $C(g)$ denotes the conjugacy class of $g$.
\end{enumerate}
\end{lemma}
From Lemma \ref{lem:trivial_hurwitz} we obtain that for $g_1,\ldots,g_n\in G$
\[
\left(\bigcup_{\tau\in S_n} C(g_{\tau(1)})\times\cdots\times C(g_{\tau(n)})\right) \cap \mu^{-1}(g_1\ldots g_n) \subset G^n
\]
is stable under the $B_n$-action.
\subsection{The Birman-Ko-Lee generators}
If $1\le i<j\le n$ then the corresponding Birman–Ko–Lee generator \cite{BKL} of $B_n$ is defined as  
\begin{equation}
\label{eq:aij}
a_{ij}=\sigma_i\cdots\sigma_{j-2}\sigma_{j-1}\sigma^{-1}_{j-2}\cdots\sigma_i^{-1}.
\end{equation}
By convention we also put $a_{ji}=a_{ij}$.
Using the braid relations we may obtain the following alternative expression for $a_{ij}$:
\begin{equation}
\label{eq:aijalt}
a_{ij}=\sigma_{j-1}^{-1}\cdots \sigma_{i+1}^{-1} \sigma_i \sigma_{i+1}\cdots \sigma_{j-2}\sigma_{j-1}.
\end{equation}
The $a_{ij}$ satisfy well-known quadratic relations  (see e.g.\ \cite[Lemma VIII.1.3]{MR2463428}). In particular if $[i,k]$, $[j,l]$ share no end points and they are either disjoint or nested then $a_{ik}$ and $a_{jl}$ commute.  In the remaining cases rather the opposite is true!
\begin{lemma}
\label{lem:norel} If $1\le i<j<k<l\le n$ then $a_{ik}$ and $a_{jl}$ generate a free subgroup of~$B_n$.
\end{lemma}
\begin{proof}
Using the mapping class group interpretation of $a_{ij}$ (see the introduction) we see that we may assume $n=4$ and $a_{ik}=a_{13}=\sigma_1\sigma_2\sigma_1^{-1}$, $a_{jl}=a_{24}=\sigma_2\sigma_3\sigma_2^{-1}$. By considering the ``exotic'' homomorphism $B_4\r B_3$ which maps $\sigma_1,\sigma_3$ to $\sigma_1$ and $\sigma_2$ to $\sigma_2$ we see
that it is sufficient to prove that $\sigma_1\sigma_2\sigma_1^{-1}$ and $\sigma_2\sigma_1\sigma_2^{-1}$ generate a free subgroup of $B_3$.
We now consider the homomorphism $\theta:B_3\r \langle u,v\mid u^3=1=v^2\rangle\cong \ZZ/3\ZZ\ast \ZZ/2\ZZ$ which sends $\sigma_1$ to $uv$ and $\sigma_2$ to $vu$.
We have $\theta(\sigma_1\sigma_2\sigma_1^{-1})=u^{-1}vu^{-1}$ and $\theta(\sigma_2\sigma_1\sigma_2^{-1})=vu^{-1}vu^{-1}v$ and it is clear that 
$u^{-1}vu^{-1}$ and $vu^{-1}vu^{-1}v$ don't satisfy any relations in the free product $\ZZ/3\ZZ\ast \ZZ/2\ZZ$.
\end{proof}
\begin{remark} Another way to prove Lemma \ref{lem:norel} is to observe that (1) $A_{ik}$ and $A_{jl}$ intersect in Figure \ref{fig:Aij}, (2) there a  well-known
relationship between half twists and Dehn twists in double covers (see \cite[\S9.4]{MR2850125}) and (3) Dehn twists corresponding to twice intersecting curves
generate a free group (see \cite[\S3.5.2]{MR2850125}).
\end{remark}
\subsection{Young subgroups of the braid group}
One has $\bar{a}_{ij}=(ij)$ and the following is an easy verification
\begin{lemma} \label{lem:birman}
Let $g\in G^n$ be such that $g_i=g_j$. Then $a_{ij}(g)=g$.
\end{lemma}
If $Q=\{Q_1,\ldots,Q_m\}$ is a partition of $\{1,\ldots,n\}$ then we define the corresponding \emph{Young subgroup} $B_Q$ of $B_n$ as the subgroup generated by those
$a_{ij}$ for which there exists an $s$ such that $i,j\in Q_s$.

If $g\in G^n$ then we let $Q^g$ be the partition of $\{1,\ldots,n\}$ such that $g_i=g_j$ if and only if $i,j\in Q^g_s$ for some $s$.
From Lemma \ref{lem:birman} we obtain that $B_{Q^g}\subset \operatorname{Stab}_{B_n}(g)$.
\begin{remark} Note that $B_Q$ is not the same as the ``mixed braid group'' \cite{MR2103472} associated to $Q$. The latter consist of those braids whose associated permutation preserves~$Q$. For example if $Q$ consists entirely of singletons then $B_Q$ is trivial whereas the mixed braid group is  the pure braid group.
\end{remark}
\subsection{Main result}
\begin{theorem} \label{th:mainth}
Let $u_1,\ldots, u_n\in \Sscr := \{x_1,  \dots, x_m\} \subset F_m$. 
\begin{enumerate}
\item \label{it:mainth1} The $B_n$-orbit of ${u=(u_1,\ldots,u_n)}$ is equal to
\begin{equation}
\label{eq:tuples}
\left(\bigcup_{\tau\in S_n} C(u_{\tau(1)})\times\cdots\times C(u_{\tau(n)})\right) \cap \mu^{-1}(u_1\ldots u_n).
\end{equation}
\item \label{it:mainth2}
The stabilizer of $u$ is equal to $B_{Q^u}$.
\end{enumerate}
\end{theorem}
Theorem \ref{th:mainth}\eqref{it:mainth1} will be proved in \S\ref{sec:transitivity} and \ref{th:mainth}\eqref{it:mainth2} will be proved in \S\ref{sec:theproof}. 
\section{Transitivity}
\label{sec:transitivity}
In this section we prove Theorem \ref{th:mainth}\eqref{it:mainth1}. The proof is basically the same as the proof of \cite[Theorem 16]{Artin}. However
Lemma \ref{lem:reduction_lemma}, which is more precise than what we need here, will be used afterwards.

\medskip

Let $\Sscr=\{x_i \mid i=1,\ldots,m\}$, $\Sscr^{\pm}=\{x_i^{\pm1}\mid 1,\ldots,m\}$ and let $\Sscr^{\pm\ast}$ be the set of words over $\Sscr^{\pm}$, including the empty word which is denoted by $1$.
If $w\in S^{\pm \ast}$ then $l(w)$ denotes the length of~$w$. By convention $l(1)=0$.

We say that a word in $\Sscr^{\pm\ast}$ is \emph{reduced}
if it does not contain any subwords of the form $x_i^{\pm 1} x_i^{\mp 1}$. We denote the set of reduced words by $\Rscr$.
Every element $w$ of $F_m$ is uniquely represented by an element $\tilde{w}$ of $\Rscr$ under the product map $\mu:\Sscr^{\pm\ast} \r F_m$. We put $l(w)=l(\tilde{w})$.

It will be convenient to write
$w=w_1{\ast}\cdots{\ast} w_n$ if $w=w_1\ldots w_n$ and $l(w)=\sum_{i=1}^n l(w_i)$.  
One observes that all parenthesized versions of the right hand side of $w=w_1{\ast}\cdots{\ast} w_n$ yield equivalent statements.\footnote{For example if we have
$w=w_1{\ast} (w_2{\ast} w_3)$ then by definition $l(w)=l(w_1)+l(w_2{\ast} w_3)=l(w_1)+l(w_2)+l(w_3)$ and so $w=w_1{\ast} w_2{\ast} w_3$.} If $f=(f_1,\ldots,f_n)\in F_m^n$ then we write $l(f):=\sum_i l(f_i)$.
\begin{lemma}
\label{lem:reduction_lemma}
Assume $u,v\in \Sscr$
and consider a pair of the form
\[
f=(C{\ast}u{\ast}C^{-1}, D{\ast}v{\ast}D^{-1})
\]
with $C, D\in F_m$.
Then
\begin{equation}
\label{eq:pair1}
l(\sigma^{-1}_1(f))\ge l(f)\text{ and } l(\sigma_1(f))\ge l(f)
\end{equation}
if and only if 
\begin{equation}
\label{eq:pair2}
\begin{aligned}
C&=Z{\ast} A,\\
D&=Z{\ast} B
\end{aligned}
\end{equation}
 for suitable $Z\in F_m$ such that
\begin{equation}
\label{eq:pair3}
CuC^{-1}DvD^{-1}=Z{\ast} A {\ast}u{\ast} A^{-1} {\ast} B{\ast} v{\ast} B^{-1}
{\ast} Z^{-1}.
\end{equation}
\end{lemma}
\begin{proof}
We claim that \eqref{eq:pair2}\eqref{eq:pair3} imply \eqref{eq:pair1}. Assume that \eqref{eq:pair1} is false. Then either $l(\sigma_1^{-1}(f))<l(f)$ or $l(\sigma_1(f))<l(f)$.
The two cases are similar, so assume that the first case holds. Hence
\begin{equation}
\label{eq:star0}
l(Cu C^{-1}  D vD^{-1} C u^{-1} C^{-1})<
l(Dv D^{-1}).
\end{equation}
Considering the cancellations in $C^{-1}D$ we find that we may write $C=Z{\ast} A$, $D=Z\ast B$ such that $C^{-1}D=A^{-1}{\ast} B$. Note that $l(A^{-1}{\ast}B)=l(C)+l(D)-2l(Z)$.
We have
\begin{equation}
\label{eq:star}
 Cu C^{-1} D v D^{-1}C u^{-1} C^{-1}=
(C{\ast} u)(A^{-1}{\ast} B) v (B^{-1}{\ast} A) (u^{-1}{\ast} C^{-1}).
\end{equation}
If there are no further cancellations then we obtain from \eqref{eq:star0}\eqref{eq:star}:
\begin{align*}
l( Cu C^{-1} D v D^{-1}C u^{-1} C^{-1})
&=l(C)+1+l(C)+l(D)-2l(Z)+1+l(C)\\ &\qquad+l(D)-2l(Z)+1+l(C)\\
&=4l(A)+2l(D)+3\\
&<l(DvD^{-1})= l(D)+1+l(D)
\end{align*}
which may be rewritten as
$
l(A)< -2
$
which is absurd.
It follows that there must be further cancellations in the right hand side of \eqref{eq:star}. Since $CuC^{-1}=C{\ast} u{\ast} C^{-1}=Z{\ast} A{\ast} u{\ast} A^{-1}{\ast} Z^{-1}$
and $DvD^{-1}=D{\ast} v{\ast} D^{-1}=Z{\ast} B{\ast} v{\ast} B^{-1}{\ast} Z^{-1}$ we see that this is only possible if $A=1$ or $B=1$. We consider two separate cases.
\begin{enumerate}
\item $B=1$. So $C=Z{{\ast}} A$ and $D=Z$. Then \eqref{eq:star0} becomes
\begin{equation}
\label{eq:strict}
(Z{\ast} A{\ast} u{\ast} A^{-1})v(A{\ast} u^{-1} {\ast} A^{-1}{\ast} Z^{-1})< l(Z{\ast} v{\ast} Z^{-1}).
\end{equation}
For the middle part to cancel we must have that $A$ is a power of $v$ (possibly $A=v^0=1$) and $v=u$.
But then we get an equality in \eqref{eq:strict} rather than a strict inequality. Hence a contradiction.
\item $A=1$, $B\neq 1$.
The right hand side of \eqref{eq:star} now becomes
\[
(C{\ast} u)(B{\ast} v {\ast} B^{-1}) (u^{-1}{\ast} C^{-1}).
\]
Further cancellation (which must occur as shown above) will happen if and only if $B=u^{-1}{\ast} B'$. But this contradicts \eqref{eq:pair3}.
\end{enumerate}

If  \eqref{eq:pair2}\eqref{eq:pair3} does not hold then by considering
the cancellation at $C^{-1}D$ in the left hand side of \eqref{eq:pair3} one sees that one of the following must be true:
\begin{equation}
\label{eq:left_eating}
D=C{\ast}u^{-1}{\ast} Y
\end{equation}
for suitable $Y$ in $F_m$, or
\begin{equation}
\label{eq:right_eating}
C=D{\ast} v{\ast} X
\end{equation}
for suitable $X$ in $F_m$. One now checks that \eqref{eq:left_eating} implies $l(\sigma^{-1}_1(f))<l(f)$ and \eqref{eq:right_eating} implies $l(\sigma_1(f))<l(f)$, finishing 
the proof.
\end{proof}
\begin{proof}[Proof of Theorem \ref{th:mainth}\eqref{it:mainth1}]
We consider the non-obvious inclusion. Assume that $t=(t_1,\ldots,t_n)=(C_1{\ast} x_{i_1}{\ast}C_1^{-1},\ldots, C_n{\ast} x_{i_n}{\ast} C_n^{-1})$ is in \eqref{eq:tuples}. We use
induction on $l(t)$ to prove that~$t$ is in the $B_n$-orbit of $(u_1,\ldots,u_n)$. If $l(t)=n$ then $t=(u_1,\ldots,u_n)$ and there is nothing to prove. So assume $l(t)>n$.
Since
\[
(C_1{\ast} x_{i_1}
{\ast}C_1^{-1})(C_2{\ast} x_{i_2}{\ast}C_2^{-1})\cdots(C_n{\ast} x_{i_n}{\ast} C_n^{-1})=u_1\cdots u_n
\]
it follows that \eqref{eq:pair2}\eqref{eq:pair3} must fail for some pair $(C_j x_{i_j} C_j^{-1},C_{j+1} x_{i_{j+1}} C_{j+1}^{-1})$. This implies that $l(\sigma_j^\epsilon(t))<l(t)$
for $\epsilon\in \{\pm 1\}$. By induction $\sigma_j^\epsilon(t)$ is in the $B_n$-orbit of $(u_1,\ldots,u_n)$. But then this is the case for $t$ as well.
\end{proof}

\section{Arc diagrams}
\subsection{Arc diagrams}
\label{sec:arcdiagram}
Below we consider the surface $\RR^2$ with a finite set of marked points~$S$ on the horizontal axis. We will also consider the point ``at infinity'' $\infty:=(0,+\infty)$.  We equip $\RR^2$ with the clockwise orientation.
\begin{definition} Let $\Sscr=\{x_1,\ldots,x_m\}$ be a set of symbols (or ``colors'').
An \emph{arc diagram $P$  labeled by $\Sscr$}
is an oriented one-dimensional submanifold with boundary of $\RR^2$ such that $\partial P\subset S$, $P\cap S\subset \partial P$
which, outside a compact subset, is of the form $\bigcup_{i=1}^t \{a_i\}\times [b_i,+\infty[$ for  $a_i,b_i\in \RR$ and such that every connected component of $P$ is assigned
an element of $\Sscr$.
\end{definition}
\subsection{Arc diagrams for words in free groups}
\label{sec:arcdiagramfree}
If $w\in \Sscr^{\pm\ast}$ then a \emph{reduction sequence} for $w$ is a sequence of words
\[
w=w_0\r w_1\r \cdots\r w_t
\]
where $w_{k+1}$ is obtained from $w_k$ by replacing a subword $x_i^{\pm 1} x_i^{\mp 1}$ by $1$ and where $w_t$~is reduced (see \S\ref{sec:transitivity}). Any symbol $x_i^{\pm 1}$ in $w$ either survives in $w_t$, in which case we say it \emph{remains single}, or else it is annihilated by some other symbol $x_i^{\mp 1}$ in $w$ in some $w_k$ which we call the \emph{partner} of $x_i^{\pm 1}$.
The following lemma gives an easy method to recognize a potential partner.
\begin{lemma}
\label{lem:partner}
Assume $w=v_1 x_i^{\pm 1} v_2 x_i^{\mp 1} v_3$ for $v_i\in \Sscr^{\pm \ast}$. Then $x_i^{\pm 1}$ is a partner for $x_i^{\mp 1}$ in some reduction sequence for $w$ if and only if $\mu(v_2)=1$.
\end{lemma}

We associate an arc diagram $P$ labeled by $\Sscr$ to a reduction sequence for $w$ with the set $S$ of marked points corresponding to the symbols in $w$.
With this convention~$P$ consists of a set of {arc}s and vertical half lines in the upper half plane connecting each symbol either with its partner, or else with $\infty$ if it stays single. 
Depending on the situation, the intervals are oriented from $x_i$ to $x_i^{-1}$, from $x_i$ to $\infty$ or from $\infty$ to~$x_i^{-1}$. We label each segment by the
 element of $\Sscr$ it is associated with.
We illustrate this for the reduction sequence
\[
x_1 x^{-1}_2 x_2  x_1^{-1}x_3\r x_1x_1^{-1}x_3\r x_3
\] 
whose corresponding arc diagram is given in Figure \ref{fig:diagram}.
\begin{figure}
\diagram
\caption{}
\label{fig:diagram}
\end{figure}

Note
that an arc diagram depends on the full reduction sequence and not only on the
initial word as we can see by considering the reduction sequences
\begin{align*}
\underline{xx^{-1}}x&\r x,\\
x\underline{x^{-1}x}&\r x.
\end{align*}
However, the following lemma is easy to prove:
\begin{lemma}
Arc diagrams corresponding to different reduction sequences of the same word can be transformed into each other using the ``mutation'' operations given in Figure \ref{fig:transform} where
the green dashed intervals represent barriers that are not intersected by other arcs.
\begin{figure}
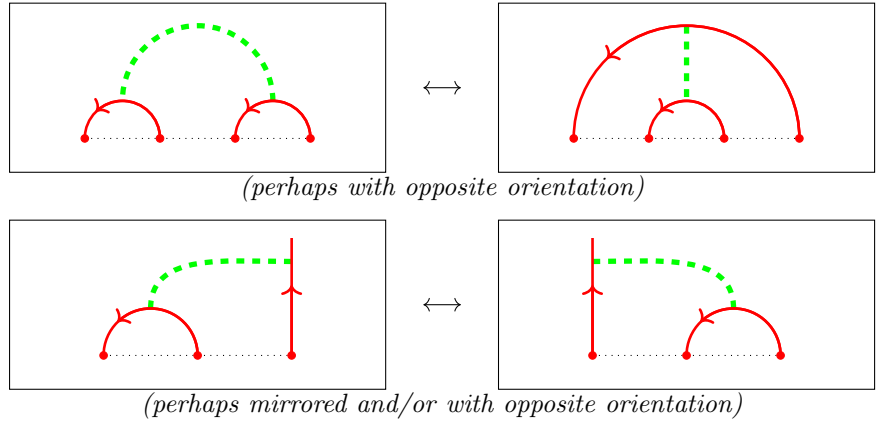

\transformi
\\
(perhaps with opposite orientation)\\[2mm]
\transformii
\\
(perhaps mirrored and/or with opposite orientation)
\caption{Mutations of reduction sequences}
\label{fig:transform}
\end{figure}
\end{lemma}
\subsection{Arc diagrams associated with  sequences of  conjugates of generators  of free groups}
Let $t=(t_1,\ldots,t_n)=(A_1x_{i_1}A_1^{-1},\ldots, A_n x_{i_n} A_n^{-1})\in F_m^n$.
We identify $t_i$ with the corresponding reduced
word in $\Sscr^{\pm\ast}$ and we consider the concatenated word $\tilde{t}=t_1\cdots t_n\in \Sscr^{\pm\ast}$. We associate an arc diagram $P$ labeled by $\Sscr$ to $t$ as follows.
\begin{itemize}
\item We ``pre-mark'' points on the horizontal axis, labeled by the symbols in $\tilde{t}$.
\item The actual marked points are the ones labeled by $(x_{i_j})_{j=1,\ldots,n}$ in $t$ (i.e.\ the middle symbols of the $(t_j)_j$).
\item $P$ is the union of an arc diagram for $\tilde{t}$ (after choosing a reduction sequence for $\tilde{t}$) and the set of {arc}s in the lower half plane connecting the corresponding symbols in $A_k$ and $A_k^{-1}$ with the  orientation going from $x_i^{-1}$  to $x_i$. 
\item We label every path in $P$ by the element of $\Sscr$ it is associated with.
\end{itemize}
Below we refer to the {arc}s in the upper and lower half plane as \emph{upper and lower {arc}s}.
\begin{example} For
\[
\def\cdot{}
t=(x_2x_1\cdot x_2\cdot x_1^{-1}x_2^{-1}, x_2, x_1, x_2^{-1}x_1^{-1}x_2^{-1}\cdot x_1\cdot x_2x_1x_2)
\]
the corresponding arc diagram is given in Figure \ref{fig:paths}.
\begin{figure}
\begin{minipage}[c]{0.49\textwidth}
\paths
\caption{}
\label{fig:paths}
\end{minipage}
\hfill
\begin{minipage}[c]{0.49\textwidth}
\pathsother
\caption{}
\label{fig:pathsother}
\end{minipage}
\end{figure}
In this example the  arc diagram has only half lines. This is not the most general situation as one sees by considering
\[
t=(x_1, x_1x_2x_1^{-1}).
\]
In this case the arc diagram is Figure \ref{fig:pathsother} and it also has full lines.
\end{example}
In general we have half lines and full lines but no more as the following lemma shows.
\begin{lemma}
The arc diagram associated to a tuple of conjugates of generators as above has no loops or intervals.
\end{lemma}
\begin{proof} Intervals are trivially excluded.\footnote{If we would more generally consider tuples of the form $t=(A_1x_{i_1}^{\pm 1}A_1^{-1},\ldots, A_n x_{i_n}^{\pm 1} A_n^{-1})$
then we could also have intervals. If we would allow the $t_i$ to be non-reduced words then loops would occur.}
A loop if it exists cannot contain one of the~$x_{i_j}$ since these are not endpoints of lower {arcs}.
It follows that the area enclosed by a loop can also not contain an $x_{i_j}$ (since this would have to be part of a loop as well).
Furthermore the areas enclosed by loops are ordered by inclusion. Let us consider a loop whose enclosed area is minimal.
It contains a lower {arc} $\alpha$ as in Figure \ref{fig:loopi}.
\begin{figure}
\loopi
\caption{}
\label{fig:loopi}
\end{figure}
Since the area enclosed by the path cannot contain $x_{i_j}$ 
the points $p$ and $q$ must be connected to each other via a path containing
a segment like $\beta$ (with perhaps $\beta$ being outside of $\alpha$) with the area between $\alpha$ and $\beta$  inside the loop.
 However, by
the minimality hypothesis on the loop, all segments between $\alpha$ and $\beta$ must be red (differently colored segments would forcibly belong to a smaller loop). It follows that there will be some adjacent red segments with opposite
orientation. However this is impossible since we have assumed that $t_j$ (viewed as a word) is reduced.
\end{proof}
In the situation of \eqref{eq:tuples} there are also no full lines.
\begin{lemma} \label{lem:half_lines} Assume $\mu(t)$ contains no strictly negative powers of the $x_i$. Then the arc diagram of $t$ has only half lines.
\end{lemma}
\begin{proof} This is clear since a full line contains a half line oriented from $\infty$ to $x_i^{-1}$. By definition then $x_i^{-1}$ stays single in $
\tilde{t}$. But this contradicts
our assumption on $\mu(t)$.
\end{proof}
One may verify the following:
\begin{lemma}
\label{lem:red2}
A tuple $t$ as above can be uniquely reconstructed from a corresponding  arc diagram.
\end{lemma}
\begin{remark}
\label{rem:homotopy}
We  observe that it is sufficient to know an arc diagram up to ambient isotopy to
reconstruct it.
 To see this we note that homotopy classes of paths
can be uniquely represented in a  minimal (!) way by segments consisting of upper and lower {arc}s, as well as vertical segments. In Figure \ref{fig:minimal} we have on the left a standard
arc diagram and on the right a minimal version. 
It is clear  how to go from the picture on the right to the picture on the left.
\begin{figure}
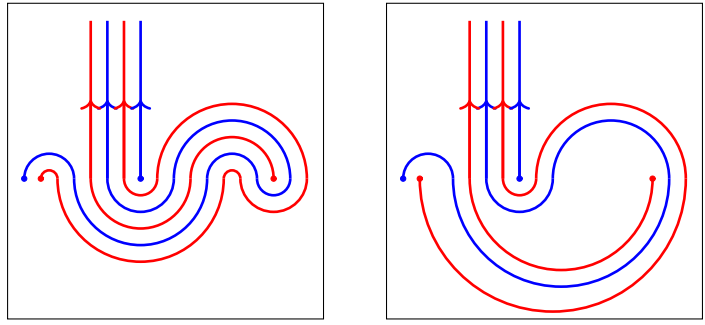

\reducedi \qquad \reducedii
\caption{An arc diagram (left) and its minimal version (right)}
\label{fig:minimal}
\end{figure}
\end{remark}
We also note the following. 
\begin{lemma} \label{lem:uniqueness}
Assume $m=n$. Then the  arc diagram associated to $t$ does not depend on the reduction sequence chosen for $\tilde{t}$.
\end{lemma}
\begin{proof} If $m=n$ then the segments in the diagrams in the rows of Figure \ref{fig:transform} must belong to the same path since by the hypothesis $m=n$,
every path has a different color($=$label). One quickly verifies that this is impossible.
E.g. if in the top row in the left hand side of Figure \ref{fig:transform} the end of the first segment is connected to the start of the second segment then this creates a loop in the right hand side.
\end{proof}
\subsection{Non-self adjacency}
In this section we assume that $P$ is an arc diagram which consists entirely of half lines and whose boundary is $S$.
\begin{lemma} \label{lem:isotopy} Let $L\subset P$ be a half line with
  starting point $p\in S$.  Assume that there are distinct points
  $q,r\in L$ and an interval $I\subset \RR^2$ such that
  $I\cap P=\{q,r\}$.  Let $L'$ be a half line obtained by changing the
  segment $J$ in $L$ bounded by $q,r$ into $I$ and let $P'$ be the arc
  diagram obtained from $P$ by replacing $L$ with $L'$. Then $P$
  can be isotoped into $P'$.
\end{lemma}
\begin{proof}
Let $B$ be the region bounded by the bigon $I\cup J$. $B$ cannot contain any marked points different from $p$, since their attached half lines would not be able to escape
from $B$.

If $B$ does not contain any marked points then it is clear that we can deform $J$ into $I$ and hence $L$ into $L'$. The case where $p\in B$ is illustrated in Figure \ref{fig:isotopy}.
\begin{figure}
\isotopy
\caption{}
\label{fig:isotopy}
\end{figure}
\end{proof}
Assume that $C_1,C_2$ are either both upper or both lower distinct arcs contained in $P$. We say that $C_1$ and $C_2$ are \emph{adjacent} if there is an interval $I\subset \RR^2$,
not intersecting the horizontal axis, which connects
$C_1$ and $C_2$ and does not intersect $P$ elsewhere. 
\begin{lemma} (Non-self adjacency) 
\label{lem:nonadjacency} Let $L\subset P$ be a half line represented in a minimal way as in Remark \ref{rem:homotopy}. Then $L$ does not contain adjacent arcs.
\end{lemma}
\begin{proof} Assume the lemma is false so that $I$ exists as described above. By Lemma \ref{lem:isotopy}, $L$ can be isotoped into a half line $L'$ containing $I$. One verifies with some quick sketches (there are several cases to consider) that $L'$ will have fewer arcs than $L$ in its minimal representation. This contradicts the hypothesis that the representation of $L$ was already minimal.
\end{proof}
\subsection{Pairs}
\label{sec:pairs}
\begin{lemma}
\label{lem:normalform} Assume $m=n$ and 
let $f=(h_i,h_{i+1})$ be a subpair of an element $h$ in \eqref{eq:tuples}.
Then there exist $Z,A,B\in F_m$ and $p,q\in \Sscr$ such that
\begin{align*}
X:=ApA^{-1}&= A{\ast}p{\ast} A^{-1},\\
Y:=BqB^{-1}&= B{\ast}q{\ast} B^{-1}
\end{align*}
and $f$ has one of the following forms
\begin{equation}
\label{eq:normal_form}
{\small
\begin{gathered}
Z(X,Y)Z^{-1},\\
Z(\underbrace{Y^{-1}X^{-1}\cdots Y^{-1}X^{-1}}_k Y \underbrace{XY\cdots XY}_k,   \underbrace{Y^{-1}X^{-1}\cdots Y^{-1}X^{-1}}_k  Y^{-1}XY \underbrace{XY\cdots XY}_k)Z^{-1},\\
Z((\underbrace{Y^{-1}X^{-1}\cdots Y^{-1}X^{-1}}_k  Y^{-1}XY \underbrace{XY\cdots XY}_k,\underbrace{Y^{-1}X^{-1}\cdots Y^{-1}X^{-1}}_{k+1} Y \underbrace{XY\cdots XY}_{k+1})Z^{-1},\\
Z((\underbrace{XY\cdots XY}_k XYX^{-1} \underbrace{Y^{-1}X^{-1}\cdots Y^{-1}X^{-1}}_k, \underbrace{XY\cdots XY}_k X \underbrace{Y^{-1}X^{-1}\cdots Y^{-1}X^{-1}}_k)Z^{-1},\\
Z( \underbrace{XY\cdots XY}_{k+1} X \underbrace{Y^{-1}X^{-1}\cdots Y^{-1}X^{-1}}_{k+1},\underbrace{XY\cdots XY}_k XYX^{-1} \underbrace{Y^{-1}X^{-1}\cdots Y^{-1}X^{-1}}_k)Z^{-1}
\end{gathered}
}
\end{equation}
where all products are $\ast$'s, $k\ge 0$ and where we have used the notation $Z(f_1,f_2)Z^{-1}$ as a shorthand for $(Zf_1Z^{-1},Zf_2Z^{-1})$.
\end{lemma}
\begin{proof} To start we will only use that $f$ is a pair as in Lemma \ref{lem:reduction_lemma}.
Choose $m\in \ZZ$ such that $l(\sigma_i^m(f))$ is minimal. Then by Lemma \ref{lem:reduction_lemma} we have from \eqref{eq:pair2}
\begin{equation}
\label{eq:sigma}
\sigma_i^m(f)=Z( A p A^{-1},B q B^{-1})Z^{-1}.
\end{equation}
Applying $\sigma_i^{-m}$ to the right hand side of \eqref{eq:sigma} yields the expressions in \eqref{eq:normal_form}. The~4 subcases correspond to
the sign and the parity of $m$.

It remains to prove that the products in \eqref{eq:normal_form} are $\ast$'s. For the products involving $Z$ as well as the product $XY$ this follows from \eqref{eq:pair3}.
Sadly we can draw no immediate conclusion about $YX$ if $A=1$ or $B=1$ (consider the pair $(x,xyx^{-1})$). We claim however that this case cannot occur if we use all the
hypotheses in the statement of the current lemma.

Assume first $A=1$ and furthermore $YX\neq Y{\ast} X$. Then $B=p{\ast} W$ so that
\begin{equation}
\label{eq:secondcomponent}
f=(ZpZ^{-1}, ZpWqW^{-1}p^{-1}Z^{-1}).
\end{equation}
To analyze this case we give the corresponding arc diagram in Figure \ref{fig:nonecase}
(using the criterion given by Lemma  \ref{lem:partner}).  We also
use Lemma \ref{lem:uniqueness} which states that the  arc diagram is unique.
\begin{figure}
\nonecase
\caption{}
\label{fig:nonecase}
\end{figure}
Note that this is a subdiagram of the arc diagram of $h$. The fat paths correspond
to bundles of parallel paths (not necessarily oriented in the same direction). The vertical lines are connected with the rest of the arc diagram
of $h$. 

From the hypothesis $n=m$ as well as Lemma \ref{lem:half_lines}
it follows that the segments $\alpha$ and $\beta$ belong to the same path. Since there are no paths in the hatched area, $\alpha$ and $\beta$ contain adjacent arcs
contradicting Lemma \ref{lem:nonadjacency}.

The case $B=1$ is similar.
\end{proof}
\subsection{Arc diagrams associated to pairs}
We discuss the arc diagrams of subpairs of elements of \eqref{eq:tuples}. To this end we use Lemma \ref{lem:normalform} which asserts that they have the forms listed
in \eqref{eq:normal_form}. To understand what happens we give in Figure \ref{fig:firstfour} arc diagrams for four of the possibilities in \eqref{eq:normal_form}, i.e. 
$\sigma_i^m(Z(ApA^{-1},BqB^{-1})Z^{-1})$ for $m=0,\ldots,3$.
We use the same conventions as for Figure \ref{fig:nonecase}.
\begin{figure}
\pathszero\ \pathsone\ \pathstwo\ \pathsthree
\caption{}
\label{fig:firstfour}
\end{figure}
A clear  pattern emerges, and using some sketches the reader can verify from the explicit formulas in \eqref{eq:normal_form} that it persists for higher $m$.  
Similarly one checks that for negative~$m$ one obtains similar diagrams.
\section{Compatibility with the braid group action}
\subsection{Half twists}
\label{sec:halftwist}
Let $I\subset M$ be an interval in an oriented smooth surface $M$ with marked points $S$ such that  $I\cap S=\{p,q\}=\partial I$.
The corresponding \emph{half twist} $H_I:(M,S)\r (M,S)$ is defined (up to isotopy) as in Figure \ref{fig:halftwist}.
\begin{figure}
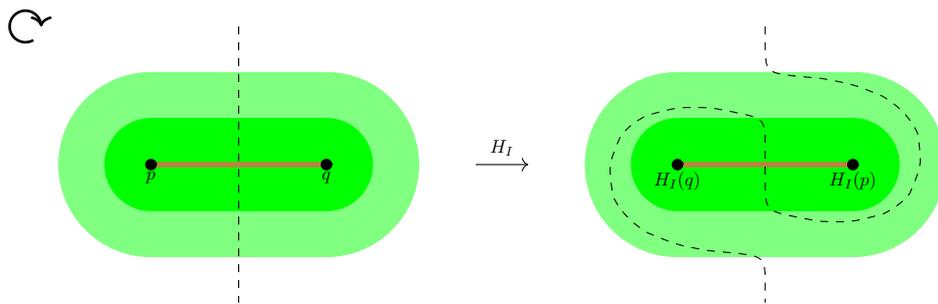

\halftwist
\caption{The half twist}
\label{fig:halftwist}
\end{figure}
The green areas are nested small thickening of the curve $I$, intersecting $S$ only in $p$, $q$. They are diffeomorphic to disks. In the dark green area the half twist is a rotation by $\pi$ in the direction of the orientation and in the white area it is the identity. The light green area interpolates between
these two extremes.
\begin{remark} Note that to define $H_I$ we do not have to choose an orientation on $I$.
\end{remark}

Assume that $\pi:M\r M$ is a diffeomorphism preserving $S$. Then it is easy to see that
\begin{equation}
\label{eq:conjHI}
\pi H_I \pi^{-1}=H_{\pi\circ I}.
\end{equation}
\subsection{The braid group and the mapping class group}
\label{sec:mapping}
We use the notations and conventions from \S\ref{sec:arcdiagram}. The marked points in $S$ will be denoted by $p_1,\ldots,p_n$ and they are assumed to be ordered by their
horizontal coordinate.
The mapping class group of $(\RR^2,S)$ will be the set of orientation preserving diffeomorphisms $\RR^2\r \RR^2$ which are the identity outside a compact set and which 
preserve $S$.
We will identify this mapping class group with the braid group~$B_n$ in such a way
that the generator $\sigma_k\in B_n$ corresponds to $H_{\Sigma_k}$ with $\Sigma_k$ being the interval
 connecting $p_k$ to $p_{k+1}$. With this convention $a_{ij}\in B_n$ (see \eqref{eq:aijalt}) is equal to $H_{A_{ij}}$ 
where $A_{ij}$ is the lower {arc} connecting $p_i$ to $p_j$ as in Figure \ref{fig:Aij}.
As usual we follow the convention $A_{ij}=A_{ji}$.
\subsection{Compatibility}
\begin{proposition}
\label{prop:compat}
Assume $m=n$. If $h$ is an element of \eqref{eq:tuples} and $P\subset \RR^2$ is its corresponding arc diagram then the diagram of $\sigma_i^{\pm 1} t$ is $\sigma_i^{\pm 1}P$. 
\end{proposition}
\begin{proof} It follows from the discussion in \S\ref{sec:pairs} that $t$ and $P$ are modified in the same way by applying $\sigma_i^{\pm 1}$ (note that by Remark \ref{rem:homotopy} it is enough to know paths up to ambient isotopy in an arc diagram). 
\end{proof}
\begin{remark} Proposition \ref{prop:compat} is of course expected but it is not a complete tautology and the condition $h \in$ \eqref{eq:tuples} is necessary.  In Figure \ref{fig:pathsother} the arc diagram of $t_1:=(x_1,x_1x_2x_1^{-1})$
was given. Putting 
\[
t_2=\sigma_1^{-1} t_1=(x_1^2x_2x_1^{-2}, x_1)
\]
and 
\[
t_3=\sigma_1^{-1}t_2= (x_1^2x_2x_1x_2^{-1}x_1^{-2}, x_1^2x_2x_1^{-2})
\]
we see that the corresponding arc diagrams are as in Figure \ref{fig:nontautology}.
\begin{figure}
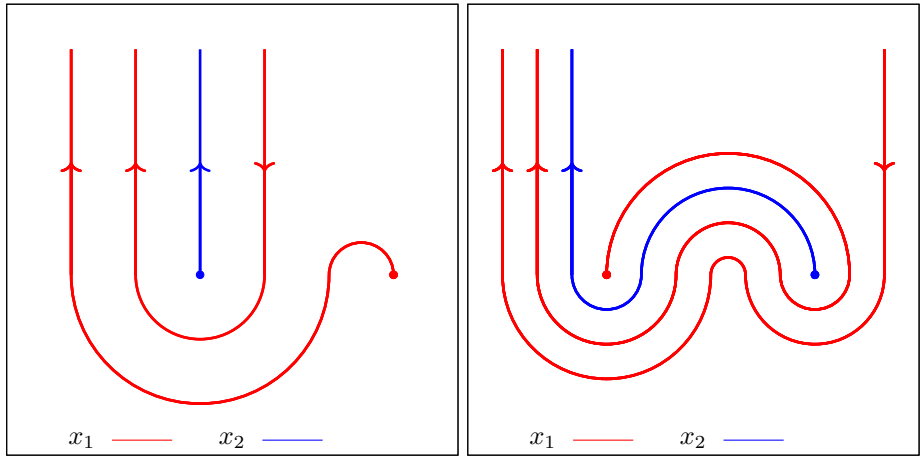

\pathsotherii \pathsotheriii
\caption{Non-compatibility with the braid group action}
\label{fig:nontautology}
\end{figure}
While the arc diagram of $t_2$ is indeed obtained by applying $\sigma_1^{-1}$ to
the arc diagram of $t_1$, we see that the arc diagram of $t_3$ is \emph{not} obtained by applying $\sigma_1^{-1}$ to
the arc diagram of~$t_2$.
\end{remark}
\def\triv{\operatorname{triv}}
\begin{corollary}
\label{cor:arc}
Assume $m=n$ and let $P_{\triv}$ be the diagram which is the union of the straight half lines, labeled by $x_i$, connecting $p_i$ to $\infty$ for $i=1,\ldots,n$.
 If $b\in B_n$ then the arc diagram of $b(x_1,\ldots,x_n)$ is
equal to $b(P_{\triv})$.
\end{corollary}
\begin{proof} This follows Proposition \ref{prop:compat} since $P_{\triv}$ is the arc diagram corresponding to $(x_1,\ldots,x_n)$. 
\end{proof}
\begin{corollary}
Assume $m=n$. Every system of non-intersecting paths connecting $(p_i)_i$ to $\infty$ is the arc diagram of an element of \eqref{eq:tuples}.
\end{corollary}

\section{Proof of Theorem \ref{th:mainth}\eqref{it:mainth2}}
\label{sec:theproof}
\subsection{The effect of half twists on paths}
\label{sec:effect}
Suppose $L_{1}, L_2\subset\RR^2$ are two non-intersecting half line paths in $\RR^2$, converging to a point at infinity.  Assume there are points $q_i\in L_i$ which are connected by an interval $I$ not intersecting $L_1\cup L_2$ anywhere else. Assume furthermore that the orientations are as in Figure \ref{fig:starthalftwist}.
\begin{figure}
\halftwistelementaryi
\caption{}
\label{fig:starthalftwist}
\end{figure}
I.e.\ if we walk on $I$ towards $L_i$ and then we walk on $L_i$ in the direction of the orientation of $L_i$ (i.e.\ towards $\infty$) then we perform a turn in the direction of the orientation of $\RR^2$ (clockwise by convention).

We may deform $I$ to a path $\tilde{I}$ connecting the starting points of $L_1$ and $L_2$ as in Figure \ref{fig:halftwistelementaryii}.
\begin{figure}
\halftwistelementaryii
\caption{}
\label{fig:halftwistelementaryii}
\end{figure}
We see that the effect of the half-twist $H_{\tilde{I}}$ consists in ``reconnecting'' the end points of $L_1$ and $L_2$ as in Figure \ref{fig:halftwistsub}.
\begin{figure}
\halftwistsub
\caption{}
\label{fig:halftwistsub}
\end{figure}
But now we see that we could have done this reconnection immediately in Figure \ref{fig:starthalftwist} without constructing $\tilde{I}$ first. I.e.\ like in Figure \ref{fig:halftwistdirect}.
\begin{figure}
\halftwistdirect
\caption{}
\label{fig:halftwistdirect}
\end{figure}

In Figure \ref{fig:starthalftwist} $I$ does not pass through the starting points of $L_1$, $L_2$. If it would e.g.\ pass through the starting point of $L_1$ then we
would get Figure \ref{fig:limiting} as a limiting case.
\begin{figure}
\halftwistdirecttwo
\caption{}
\label{fig:limiting}
\end{figure}

If instead of the orientation in Figure \ref{fig:starthalftwist} we start with the opposite orientation, as in Figure \ref{fig:halftwistelementaryiii}
\begin{figure}
\halftwistelementaryiii
\caption{}
\label{fig:halftwistelementaryiii}
\end{figure}
then we obtain by the analogous  procedure the effect of $H^{-1}_{\tilde{I}}$.
\subsection{The effect of half twists on collections of paths}
\label{sec:effectcollections}
We consider again $(\RR^2,p_1,\ldots,p_n)$ as a marked surface as in \S\ref{sec:mapping}. We let 
$P=(L_i)_{i=1,\ldots,n}$ be a collection of non-intersecting
paths in $\RR^2$ starting in $p_i$ and converging to $\infty$. We let $b\in B_n$ be such that $P=b(P_{\triv})$.

We first define a measure of distance between $P$ and the trivial collection~$P_{\triv}$. To do this we note that every $L_i$ can be drawn as a minimal collection of upper and lower {arc}s, as well as a single vertical half line.  See the
picture on the right in Figure \ref{fig:minimal}. 

We will also consider each $p_i$ that is not the end point of such an
{arc} as defining its own ``infinitesimal'' lower {arc} (i.e.\
with radius $0<\epsilon\ll 1$).  The lower
{arc}s, including the
infinitesimal ones, divide the region below the horizontal axis into a
number of regions (some infinitesimal). We define $c(P)=\sum_r c_r$
where~$r$ runs over these regions and where $c_r$ is the number of
{arc}s containing the region~$r$. In other words $c(P)$ keeps
track of how much ``overlapping'' there is among the lower {arc}s.  We
call $c(P)$ the \emph{complexity} of a minimal arc diagram $P$.  Clearly $c(P)=n$ if and only if $P$ is trivial.
\begin{example} Consider the arc diagram on the right in Figure \ref{fig:minimal}.  Its lower half consists of 6 nested {arc}s (one infinitesimal), defining 7 regions (one infinitesimal). 
The corresponding complexity is $0+1+2+3+4+5+6=21$.
\end{example}
We now show that the complexity of $P$ can be reduced by applying
suitable half twists. Assume there are two lower {arc}s which are
\emph{adjacent} in the sense that there is path $I$ entirely below the
horizontal axis which connects them, and which does not intersect any
other paths in $P$, as illustrated in Figure \ref{fig:adj1}. Let $\tilde{I}$ be a path obtained by deforming $I$ as in \ref{sec:effectcollections}. 
\begin{figure}
\adjacent
\caption{}
\label{fig:adj1}
\end{figure}
We have labeled the arcs with $p_i$, $p_j$ with the convention that the path labeled $p_i$ is starting in $b(p_i)$.  Observe that $\tilde{I}$ is connecting $b(p_i)$ and $b(p_j)$. In the picture we have chosen a particular orientation on the 
paths but if the orientations go in the opposite direction then the discussion below remains valid on condition that $H_{\tilde{I}}$ is replaced by $H^{-1}_{\tilde{I}}$.

As
explained in \S\ref{sec:effect} the result of applying $H_{\tilde{I}}$ to $P$ is as in Figure \ref{fig:adj2}.
\begin{figure}
\adjacenttwo
\caption{}
\label{fig:adj2}
\end{figure}
Since the {arc}s in Figure \ref{fig:adj2} overlap less than in Figure \ref{fig:adj1}, the complexity has been reduced. Note that in Figure \ref{fig:adj2} the {arc}s we have drawn do not 
necessarily correspond to a minimal representation of $H_{\tilde{I}}$, so the reduction in complexity may be even greater than what Figure \ref{fig:adj2} suggests.

In Figure \ref{fig:adj1} we have omitted the degenerate case where the blue {arc} is the point corresponding to $p_j$. This case is treated via Figure \ref{fig:limiting}. 
We leave it to the reader to check that the
conclusion is the same.

\medskip

Since $\RR^2-P$ is simply connected, $\tilde{I}$ is unique up to isotopy.  If follows that $\tilde{I}$ is 
isotopic to $b(A_{ij})$. From \eqref{eq:conjHI} we obtain
\[
H_{\tilde{I}}=b\circ H_{A_{ij}} \circ b^{-1}=b\circ a_{ij}\circ b^{-1}
\]
so that we obtain the formula
\[
H_{\tilde{I}}(P)=(b\circ a_{ij})(P_{\triv}).
\]

\subsection{The proof}
\label{ssec:theproof}
Let the notations be as in the statement of Theorem \ref{th:mainth}\eqref{it:mainth2} and let $b\in B_n$ be in the stabilizer of $u:=(u_1,\ldots,u_n)$. We need to prove
that $b\in B_{Q^u}$.  Let $(v_1,\ldots,v_n)$ be independent variables and put $F_n=\langle v_1,\ldots,v_n\rangle$ (thus in this context $m=n$).
Let $t_{\triv}:=(v_1,\ldots,v_n)$  and put $t=b(t_{\triv})$, $P=b(P_{\triv})$. Then according to Corollary \ref{cor:arc} $P$ is the arc diagram of $t$.
If $l(t)=n$ then $t=t_{\triv}$. It is well-known that the stabilizer of $t_{\triv}$ in $F_n^n$ is trivial (see e.g. \cite{MR2463428},  XI, Corollary 1.8). Thus we have $b=1$ and hence $b\in B_{Q^u}$.

Hence assume $l(t)>n$. We will now perform induction on $c(P)$ under this assumption. If $c(P)=n$ (the minimal value) then $P=P_{\triv}$ and hence $b=1$ which contradicts $l(t)>n$. So assume $c(P)>n$.
Since $l(b(u))=n$ the hypothesis $l(t)>n$ implies that there must be some cancellation in $t$ after performing the substitution $v_i\r u_i$. In other words:
some $t_k$ contains a subword of the form $v_i ^{\eta} v_j^{-\eta}$ with $u_i=u_j$ and 
$\eta\in \{\pm 1\}$.

This subword yields two adjacent lower {arc}s  with opposite orientation in the minimal representation of $P$, as in \S\ref{sec:effect}. It follows 
from the discussion in \S\ref{sec:effectcollections} 
 that for $b'=b\circ a_{ij}^\eta$, $P':=b'(P_{\triv})$ we have $c(P')<c(P)$.
If $l(b'(v))=n$ then $b'=1$ and hence $b=a_{ij}^{-\eta}\in B_{Q^u}$. If $l(b'(v))>n$ then by induction we may assume that $b'\in B_{Q^u}$ and we again conclude $b=b'\circ 
a_{ij}^{-\eta}\in B_{Q^u}$.
\section{Algebraic version of the algorithm}
The proof given in \S\ref{ssec:theproof} is in fact constructive and yields an algorithm to write an element $b$ in the stabilizer of $(u_1,\ldots,u_n)$ as a product
of $a_{ij}$ with $u_i=u_j$. 
It can be trivially implemented on a computer algebra system such as Sage \cite{Sage}.

To describe the algorithm concretely we first we need to do some preparatory work.
\subsection{More actions}
Let $G$ be a group. Recall that the Hurwitz action on $G^n$ is obtained by identifying
\begin{equation}
\label{eq:ident}
\Hom_{\text{Groups}}(F_n,G)\cong G^n
\end{equation}
and then using the $B_n$-action on $F_n=\langle x_1,\ldots,x_n\rangle$ via the formula
\begin{equation}
\label{eq:dotprod}
(b\cdot \phi)(w)=\phi(b^{-1}w).
\end{equation}
Assume $G=F_n$. Then we obtain a second  action of $B_n$ on $\Hom_{\text{Groups}}(F_n,F_n)$, denoted by $\bullet$, via
\begin{equation}
\label{eq:bulletprod}
(b\bullet \phi)(w)=b(\phi(w)).
\end{equation}
After the identification \eqref{eq:ident} we see that for $t\in F_n^n$ we may compute $b\bullet t$ as follows:
\begin{itemize}
\item Put $t_{\triv}=(x_1,\ldots,x_n)$ and compute $s=b^{-1}\cdot t_{\triv}$. 
\item Substitute $x_i=s_i$ in $t$.
\end{itemize}
In particular
\begin{equation}
\label{eq:inparticular}
b\bullet t_{\triv}=b^{-1}\cdot t_{\triv}.
\end{equation}
\subsection{Description of the algorithm}
\label{sec:algo}
Let the notations be as Theorem \ref{th:mainth}. As in~\S\ref{sec:theproof}, we consider tuples over $F_n:=\langle v_1,\ldots,v_n\rangle$ and we put $t_{\triv}=(v_1,\ldots,v_n)$.

Let $b$ be in the stabilizer of $(u_1,\ldots,u_n)$. We will produce a sequence of elements
$(a^{\gamma_1}_{i_tj_t},\ldots,a^{\gamma_t}_{i_1j_1})$ 
with $\gamma_{i}\in \{\pm 1\}$ such
that
\[
b=a^{\gamma_1}_{i_tj_t}\cdots a^{\gamma_t}_{i_1j_1}.
\]
To do this we keep track of a tuple $t$ which is initialized as $t:=b(t_{\triv})$
as well as a partial sequence
$l=(a^{\gamma_s}_{i_sj_s},\ldots,a^{\gamma_t}_{i_1j_1})$ which is initialized as $()$.
We now look for some $t_k$ containing a subword of the form $v_i ^\eta v_j^{-\eta}$ with $u_i=u_j$ and $\eta\in\{\pm 1\}$.
We put $t:=a_{ij}^{-\eta}\bullet t$ 
and we prepend $a_{ij}^{-\eta}$ to~$l$. We repeat this computation until 
$t=t_{\triv}$.
\begin{example} We place ourselves in the context of Example \ref{ex:usual}.
Thus $n=4$ and $u=(x,y,x,y)$. Put $t_{\triv}=(x_1,y_2,x_3,y_4)$. Let
\begin{equation}
\label{eq:bfirst}
b=a_{24}a_{13}^{-1}.
\end{equation}
Then
\[
t=b\cdot t_{\triv}=(x_1x_3x_1^{-1}, y_4, y_4^{-1}x_1x_3^{-1}y_2x_3x_1x_3^{-1}y_2^{-1}x_3x_1^{-1}y_4, y_4^{-1}x_1x_3^{-1}y_2x_3x_1^{-1}y_4).
\]
The first subword we consider is
\[
x_3x_1^{-1}
\]
which yields $l:=(a_{13}^{-1})$ (recall $a_{31}=a_{13}$). 
We also get
\[
t:=a_{13}^{-1}\bullet t=
(x_1, y_4, y_4^{-1}y_2x_3y_2^{-1}y_4, y_4^{-1}y_2y_4).
\]
Now we consider the subword
\[
y_4^{-1}y_2
\]
which yields $l:=(a_{24},a_{13}^{-1})$.
We now get
\[
t:=a_{24}\bullet t=(x_1, y_2, x_3, y_4).
\]
Since $t=t_{\triv}$ the algorithm stops and we have indeed recovered \eqref{eq:bfirst}.
\end{example}

\subsection{Explanation why the algorithm works}
 We make \S\ref{ssec:theproof} explicit. We have
\begin{align*}
t'&=(b\circ a^\eta_{ij})(t_{\triv})\\
&=(b\circ a^\eta_{ij})^{-1}\bullet t_{\triv}&\text{(by \eqref{eq:inparticular})}\\
&=a^{-\eta}_{ij}\bullet (b^{-1}\bullet t_{\triv})\\
&=a^{-\eta}_{ij}\bullet t &\text{(by \eqref{eq:inparticular})}.
\end{align*}
If we repeat the procedure $t\r t'$ we construct a sequence  $(a^{-\eta_t}_{i_tj_t},\ldots,a^{-\eta_1}_{i_1j_1})$ 
such that
\begin{align*}
t_{\triv}&=(a^{-\eta_t}_{i_tj_t}\cdots a^{-\eta_1}_{i_1j_1})\bullet t\\
&=(a^{-\eta_t}_{i_tj_t}\cdots a^{-\eta_1}_{i_1j_1})\bullet (b^{-1}\bullet t_{\triv})\\
&=(a^{-\eta_t}_{i_tj_t}\cdots a^{-\eta_1}_{i_1j_1}b^{-1})\bullet t_{\triv}
\end{align*}
from which we obtain
\[
b=a^{-\eta_t}_{i_tj_t}\cdots a^{-\eta_1}_{i_1j_1}.
\]


\providecommand{\bysame}{\leavevmode\hbox to3em{\hrulefill}\thinspace}
\providecommand{\MR}{\relax\ifhmode\unskip\space\fi MR }
\providecommand{\MRhref}[2]{%
  \href{http://www.ams.org/mathscinet-getitem?mr=#1}{#2}
}
\providecommand{\href}[2]{#2}

\end{document}